\documentclass[11pt, oneside]{article}
\usepackage{amsfonts}
\usepackage{mathrsfs}
\usepackage{color}
\usepackage[colorlinks]{hyperref}
\usepackage{latexsym}
\usepackage{amssymb}
\usepackage{amsmath}
\usepackage{enumerate}
\usepackage{amsthm}
\usepackage{indentfirst}
\usepackage{mathtools}
\usepackage{tikz}
\usepackage{psfrag}
\usepackage{graphicx}
\usepackage{epstopdf}
\usepackage{float}
\usepackage{subfig}
\usepackage{bm}
\usepackage[rflt]{floatflt}
\usepackage{float}
\usepackage{authblk}
\usepackage{ulem}
\usepackage[square, comma, sort&compress, numbers]{natbib}
\usepackage{verbatim}
\usepackage{amsthm}
\usepackage{esint}

\numberwithin{equation}{section}
\allowdisplaybreaks

\newenvironment{proof2.1}{\medskip\noindent{\bf Proof of the Theorem 2.1:}\enspace}{\hfill \qed \newline \medskip}

\newenvironment{proof2.2}{\medskip\noindent{\bf Proof of the Theorem 2.2:}\enspace}{\hfill \qed \newline \medskip}

\newtheorem{theorem}{\color{black}\indent Theorem}[section]
\newtheorem{lemma}{\color{black}\indent Lemma}[section]

\newtheorem{definition}{\color{black}\indent Definition}[section]
\newtheorem{remark}{\color{black}\indent Remark}[section]

\usepackage{amssymb,amsmath}
\begin{document}
\title{Existence of nontrivial solutions to a critical Kirchhoff equation with a logarithmic
type perturbation in dimension four}
\author{Qian Zhang\qquad Yuzhu Han$^{\dag}$}

\affil{School of Mathematics, Jilin University,
 Changchun 130012, P.R. China}
\renewcommand*{\Affilfont}{\small\it}
\date{} \maketitle
\vspace{-20pt}

\footnotetext{\hspace{-1.9mm}$^\dag$Corresponding author.\\
Email addresses: yzhan@jlu.edu.cn(Y. Z. Han).

\thanks{
$^*$Supported by the National Key Research and Development Program of China
(grant no.2020YFA0714101).}}

{\bf Abstract}
In this paper, a critical Kirchhoff equation with a logarithmic type subcritical term is considered in a bounded domain in $\mathbb{R}^4$. We view this problem as a critical elliptic equation with a nonlocal perturbation, and investigate how the nonlocal term affects the existence of weak solutions to the problem. By means of Ekeland's variational principle, Br\'{e}zis-Lieb's lemma and some convergence tricks for nonlocal problems, we show that this problem admits a local minimum solution and a least energy solution under some appropriate assumptions on the parameters. Moreover, under some further assumptions, the local minimum solution is also a least energy solution. Compared with the ones obtained in \cite{DengHe2023} and \cite{Hajaiej2024}, our results show that the introduction of the nonlocal term enlarges the ranges of the parameters such that the problem admits weak solutions, which implies that the nonlocal term has a positive effect on the existence of weak solutions.

{\bf Keywords}  Kirchhoff equation;  Critical; Logarithmic type perturbation;
Ekeland's variational principle; Br\'{e}zis-Lieb's lemma.

{\bf AMS Mathematics Subject Classification 2020:} Primary 35D30; Secondary 35J60.

\section{Introduction}

In this paper, we consider the following Kirchhoff type elliptic problem with a logarithmic subcritical term
\begin{equation}\label{eq1}
\begin{cases}
-(1+b\int_{\Omega}|\nabla{u}|^2\mathrm{d}x) \Delta{u}=\lambda u+\mu u\ln u^{2}+|u|^{2}u,&x\in\Omega,\\
u(x)=0,&x\in\partial\Omega,
\end{cases}
\end{equation}
where $\Omega\subset \mathbb{R}^4$ is a bounded domain with smooth boundary $\partial\Omega$,
$b>0$, $\mu<0$ and $\lambda\in\mathbb{R}$.

The equation in \eqref{eq1} is closely related to the stationary counterpart of the following
wave equation
\begin{align}\label{eq2}
\rho\frac{\partial^2u}{\partial{t^2}}-\left(\frac{P_0}{h}+\frac{E}{2L}\int_{0}^{L}\left|\frac{\partial u}{\partial x}\right|^2\mathrm{d}x\right)\frac{\partial^2u}{\partial x^2}=f(x,u),
\end{align}
where $\rho$, $P_0$, $h$, $E$, $L$ represent some physical quantities. The nonlocal term appears
in \eqref{eq2} as a consequence of taking into account the effects of changes in string length
during the vibrations. Since equation \eqref{eq2} was first proposed by Kirchhoff in \cite{Kirchhoff},
such equations are usually referred to as Kirchhoff equations afterwards.

In the past few years, much effort has been devoted to the study of Kirchhoff type elliptic problems
without the logarithmic subcritical term, and many interesting results have been obtained on the
existence (possibly multiplicity) of weak solutions to such problems. In particular, for the following
second order Kirchhoff type elliptic problem
\begin{equation}\label{eq3}
\begin{cases}
-(a+b\int_{\Omega}|\nabla{u}|^2\mathrm{d}x) \Delta{u}=f(x,u),&x\in\Omega,\\
u(x)=0,&x\in\partial\Omega,
\end{cases}
\end{equation}
Naimen \cite{ND} investigated it with $f(x,u)=\lambda g(x,u)+u^5$ and $\Omega$ being a bounded smooth
domain in $\mathbb{R}^3$. By means of Mountain Pass Lemma and truncation argument,
he obtained the existence and nonexistence of solutions to problem \eqref{eq3} when $g(x,u)$ satisfies
some structural conditions and $\lambda\in \mathbb{R}$.  When $g(x,u)=u$, Zhong et al.
\cite{Zhong-Tang} proved that problem \eqref{eq3} admits at least two solutions for $\lambda$ in a small right
neighborhood of $\lambda_1(\Omega)$ with the help of Nehari manifold. Here $\lambda_1(\Omega)>0$ is the first eigenvalue
of $(-\Delta, H_0^1(\Omega))$. When the dimension of the space $N=4$, problem \eqref{eq1} with $\mu=0$ was
considered by Naimen in \cite{Naimen2}. Among many other interesting results, he showed that when
$\lambda\in(0,\lambda_1(\Omega))$, problem \eqref{eq1} admits a solution if and only if $bS^2<1$, where $S>0$
is the best embedding constant from $H_0^1(\Omega)$ to $L^4(\Omega)$. Later, Naimen \cite{Naimen3}
reconsidered this problem in an open ball in $\mathbb{R}^3$ and revealed the combined effect of $\lambda$ and $b$ on the existence
and nonexistence of solutions to problem \eqref{eq1}. For more results on the existence and (or) non-existence
of solutions to critical or subcritical Kirchhoff equations, interested readers may refer to \cite{ChenKuoWu,Faraci,Figueiredo2013,LiHan2023,Perera,Shuaiwei,TangCheng} and their references.

We remark that when dealing with problem \eqref{eq3} in the framework of variational methods,
the cases when $N\leq 3$ and $N\geq 4$ turn out to be quite different. Indeed, when one looks for weak solutions
to \eqref{eq3} by using variational methods, $f$ is usually required to satisfy the so-called
Ambrosetti-Rabinowitz condition: i.e., for some $\theta >4$ and $R>0$,  there holds
\begin{equation}\label{A-R}
0<\theta F(x,t)\leq tf(x,t), \ \forall~|t|>R,~x\in \Omega,
\end{equation}
which means that $f$ is $4$-superlinear in $t$ at infinity, that is,
  \begin{equation}\label{4-super}
\lim\limits_{t\rightarrow+\infty}\frac{F(x,t)}
{t^4}=+\infty,
\end{equation}
where $F(x,t)=\int_{0}^{t}f(x,\tau)\mathrm{d}\tau$.
Condition \eqref{A-R} is a key point when we prove the boundedness of any $(PS)$ sequence of
the corresponding energy functional. Assume in addition that $f$ satisfies
the subcritical or critical growth condition
\begin{equation}\label{subcritical}
|f(x,t)|\leq C(|t|^{q-1}+1),  \  t \in \mathbb{R}, \ x \in \Omega,
\end{equation}
where $C>0$, $2<q\leq2^*=\frac{2N}{N-2}$. Combining this with \eqref{4-super},
one has $q>4$, which, together with $q\leq2^*$, implies that $N<4$. Hence, the cases when $N\leq 3$ and $N\geq 4$
are usually treated separately. In particular, when $N=4$, the power corresponding to the critical term is
equal to that corresponding to the nonlocal term in the energy functional, which makes the structure of the corresponding energy functional more complicated.
Moreover, due to the uncertainty of the sign of logarithmic term and the fact that the logarithmic term does not satisfy the standard Ambrosetti-Rabinowitz condition, the methods used in the above mentioned literature cannot be applied to critical Kirchhoff equations with logarithmic type perturbations.
This requires us to develop some new ideas.

In recent years, the research on the existence of solutions to elliptic equations with logarithmic terms has been achieving some interesting results.
In particular, the following critical elliptic problem with a logarithmic type perturbation
\begin{equation*}
\begin{cases}
-\Delta{u}=\lambda u+\mu u\ln u^{2}+|u|^{2^{*}-2}u,&x\in\Omega,\\
u(x)=0,&x\in\partial\Omega
\end{cases}
\end{equation*}
was considered by Deng et al. \cite{DengHe2023}, where $\Omega\subset\mathbb{R}^N~(N\geq3)$ is a bounded domain with
smooth boundary $\partial\Omega$, $\lambda, \mu \in \mathbb{R}$ are parameters and $2^{*}=\frac{2N}{N-2}$ is the
critical Sobolev exponent for the embedding $H_0^1(\Omega)\hookrightarrow L^{2^*}(\Omega)$. With the help of Mountain
Pass Lemma and some delicate estimates on the logarithmic term, they obtained the existence of a positive mountain
pass type solution (which is also a ground state solution) when $\lambda\in \mathbb{R}$, $\mu>0$ and $N\geq4$.
The existence of a positive weak solution was also obtained under some suitable assumptions on $\lambda, \mu$
when $N=3,4$ for the case $\mu<0$. Moreover, they proved a nonexistence result when $N\geq3$. Later, these results were
extended to critical fourth order elliptic problem with a logarithmic type perturbation by Li et al. \cite{LHW2023}.
Recently, the authors of this paper \cite{ZhangHanWang2023} weakened part of existence condition
for the case $\mu<0$ in \cite{LHW2023} and specified the types and the energy levels of the solutions by using
Br\'{e}zis-Lieb's lemma and Ekeland's variational principle. For more results on the existence and multiplicity of
solutions to local or nonlocal elliptic equations with logarithmic type nonlinearities, we refer the interested
readers to \cite{LiHan2023,Shuaiwei, Tianshuying2017, Shuaiwei2, Gao2023,Squassina, Wang2019} and the references therein.

Inspired by the above works described,  the object of this paper is to consider a critical Kirchhoff type elliptic problem with a logarithmic type subcritical perturbation  in a bounded domain in $\mathbb{R}^4$, i.e., problem \eqref{eq1}. We view this problem as an elliptic equation perturbed by a nonlocal term and
investigate how the nonlocal term affects the existence of weak solutions to problem \eqref{eq1}. It is worth noting that the presence of a nonlocal term, together with a critical nonlinearity and a logarithmic term prevents to apply in a straightforward way the classical critical point theory.
To be a little more precise, the critical term makes the corresponding energy functional lack of compactness, and since the nonlocal term is involved,
it is hard to deduce from  $u_n\rightharpoonup u$ weakly in $H_0^1(\Omega)$ that the convergence
$\int_{\Omega}|\nabla u_n|^2\mathrm{d}x\rightarrow\int_{\Omega}|\nabla u|^2\mathrm{d}x$
in general. Moreover, the logarithmic nonlinearity is sign-changing and satisfies neither the
monotonicity condition nor the Ambrosetti-Rabinowitz condition, which makes the study of problem
\eqref{eq1} more challenging.

To overcome the above mentioned difficulties, we borrow some ideas from \cite{Hajaiej2024, ZhangHanWang2023,LHW2023, DengHe2023} and put together Ekeland's variational principle, Br\'{e}zis-Lieb's lemma and some careful analysis on the logarithmic term to prove that
problem \eqref{eq1} admits a local minimum solution and a least energy solution under some appropriate assumptions on $b, \lambda, \mu$.
Moreover, under some further assumptions, the local minimum solution is also a least energy solution.
Comparing our results with the ones in \cite{DengHe2023} and \cite{Hajaiej2024}, one can see
that the range of the parameters $\lambda$ and $\mu$ for problem \eqref{eq1} with $b>0$ to admit
weak solutions is larger than that with $b=0$ in dimension four, which means that the nonlocal term
plays a positive role for problem \eqref{eq1} to admit weak solutions.

The organization of this paper is as follows. In Section $2$, we introduce some notations, definitions and necessary lemmas.
The main results of this paper are also stated here. The detailed proof of the main results is given in Section $3$.

\section{Preliminaries and the main results}

We begin this section with some notations and definitions that will be used throughout the paper.
The notation $|\Omega|$ means the Lebesgue measure of $\Omega$ in $\mathbb{R}^4$.
We use $\|\cdot\|_p$ to denote the usual $L^p(\Omega)$ norm for $1\leq p\leq\infty$ and denote the
norm of the Sobolev space $H_0^1(\Omega)$ by $\|\cdot\|:=\|\nabla \cdot\|_2$. For each Banach space $B$,
we use $\rightarrow$ and $\rightharpoonup$ to denote the strong and weak convergence in it, respectively.
By $H^{-1}(\Omega)$ we denote the dual space of $H_0^1(\Omega)$ and the dual pair between $H_0^1(\Omega)$
and $H^{-1}(\Omega)$ is written as $\langle\cdot,\cdot\rangle$. We use $C$ to denote a generic positive constant
which may vary from line to line and denote by $o_{n}(1)$ an infinitesimal as $n\rightarrow \infty$.
Moreover, $\lambda_1(\Omega)$ denotes the first eigenvalue of $(-\Delta, H_0^1(\Omega))$ and $S$ denotes
the best embedding constant from $H_0^1(\Omega)$ to $L^4(\Omega)$, i.e.,
\begin{align}\label{ineq1}
\|u\|_4\leq S^{-\frac{1}{2}}\|u\|, \ \ \ \forall\,u\in H_0^1(\Omega).
\end{align}

In this paper, the solutions to problem \eqref{eq1} are considered in the following weak sense.
\begin{definition}\label{weaksolution}$\mathrm{\bf{(Weak \ solution)}}$
A function $u\in H_0^1(\Omega)$ is called a weak solution to problem \eqref{eq1}, if it holds that
$$(1+b\|u\|^2)\int_{\Omega}\nabla u\nabla \phi\mathrm{d}x-\lambda\int_{\Omega}u\phi\mathrm{d}x-\mu\int_{\Omega}u\phi\ln u^2\mathrm{d}x-
\int_{\Omega}|u|^2u\phi\mathrm{d}x=0, \ \forall\,\phi\in H_0^1(\Omega). $$
\end{definition}

In order to look for weak solutions to problem \eqref{eq1} in the framework of variational methods,
we introduce the energy functional associated with problem \eqref{eq1} and its Fr\'{e}chet derivative,
which are given, respectively, by
\begin{align}\label{functional}
I(u)=\frac{1}{2}\|u\|^2+\frac{b}{4}\|u\|^4-\frac{\lambda}{2}\|u\|_2^2+\frac{\mu}{2}\|u\|_2^2-\frac{\mu}{2}\int_{\Omega}u^2\ln{u^2}\mathrm{d}x
-\frac{1}{4}\|u\|_4^4, ~~~\forall\,u\in H_0^1(\Omega),
\end{align}
and for all $u,\phi\in H_0^1(\Omega)$,
\begin{align}\label{frechet}
\langle I'(u),\phi\rangle=(1+b\|u\|^2)\int_{\Omega}\nabla u\nabla \phi\mathrm{d}x-\lambda\int_{\Omega}u\phi\mathrm{d}x-\mu\int_{\Omega}u\phi\ln u^2\mathrm{d}x-\int_{\Omega}|u|^2u\phi\mathrm{d}x.
\end{align}
It is easily verified that the energy functional $I(u)$ is a $\mathcal{C}^1$ functional in $H_{0}^{1}(\Omega)$ and
every critical point of $I(u)$ is a weak solution to problem \eqref{eq1}.

In order to overcome the difficulties caused by the nonlocal term, the logarithmic term and the critical term,
we introduce the following three important lemmas. The first one is Br\'{e}zis-Lieb's lemma which guarantees
the convergence of the minimizing sequence, the second one is Eklend's variational principle and the third one
presents some basic inequalities that can be used to deal with the logarithmic term.

\begin{lemma}\label{lem-Lieb}(Br\'{e}zis-Lieb's lemma \cite{BreLieb})
Let $p\in(0,\infty)$. Suppose that $\{u_n\}$ is a bounded sequence in $L^p(\Omega)$ and $u_n\rightarrow u$ a.e. in  $\Omega$.
Then
\begin{eqnarray*}
&\lim\limits_{n\rightarrow\infty}(\|u_n\|_p^p-\|u_n-u\|_p^p)=\|u\|_p^p.
\end{eqnarray*}
\end{lemma}

\begin{lemma}\label{Ekeland}(Ekeland's variational principle \cite{Ekeland})
Let $V$ be a complete metric space, and $F: V\rightarrow \mathbb{R}\cup\{+\infty\}$ be a lower semicontinuous function,
not identically $+\infty$, and bounded from below. Then for every $\varepsilon,\delta>0$ and every point $v\in V$ such that
\begin{align*}
\inf\limits_{h\in V} F(h)\leq F(v)\leq\inf\limits_{h\in V} F(h)+\varepsilon,
\end{align*}
there exists some point $u\in V$ such that
\begin{align*}
&F(u)\leq F(v),\\
&d(u,v)\leq\delta,\\
F(w)>&F(u)-(\varepsilon/\delta) d(u,w), \ \ \forall\, w\in V, w\neq u,
\end{align*}
where $d(\cdot,\cdot)$ denotes the distance of two points in V.
\end{lemma}

\begin{lemma}\label{logarithmic inequality}
$(1)$ For all $t\in(0,1]$, there holds
\begin{equation}\label{log1}
|t\ln t |\leq\frac{1}{e}.
\end{equation}

$(2)$ For all $t>0$, there holds
\begin{equation}\label{log2}
t^2-t^2\ln t^2\leq1.
\end{equation}

$(3)$ For any $\delta>0$, there holds
\begin{equation}\label{log4}
\frac{\ln t}{t^\delta}\leq \frac{1}{\delta e}, \qquad \forall\ t>0.
\end{equation}
\end{lemma}

To state our main results clearly, we introduce the following sets
\allowdisplaybreaks
\begin{align*}
\mathcal{A}_1:&=\left\{( b, \lambda, \mu)|\   b=\frac{1}{S^2}, \lambda\in[0,\lambda_1(\Omega)), \mu<0 \right\},\\
\mathcal{A}_2:&=\left\{( b, \lambda, \mu)|\   b=\frac{1}{S^2}, \lambda\in(-\infty,0)\cup[\lambda_1(\Omega),+\infty), \mu<0 \right\},\\
\mathcal{A}_3:&=\left\{(b, \lambda, \mu)|\ b>\frac{1}{S^2}, \lambda\geq0, \mu<0 \right\},\\
\mathcal{A}_4:&=\left\{(b, \lambda, \mu)|\  b>\frac{1}{S^2}, \lambda<0, \mu<0\right\},\\
\mathcal{A}_5:&=\left\{( b, \lambda, \mu)|\   0<b<\frac{1}{S^2}, \lambda\in[0,\lambda_1(\Omega)), \mu<0, \frac{(\lambda_1(\Omega)-\lambda)^2S^2}{4\lambda^2_1(\Omega)(1-bS^2)}+\frac{\mu}{2}|\Omega|>0\right\},\\
\mathcal{A}_6:&=\left\{(b, \lambda, \mu)|\  0<b<\frac{1}{S^2}, \lambda\in\mathbb{R}, \mu<0, \frac{S^2}{4(1-bS^2)}+\frac{\mu}{2}e^{-\frac{\lambda}{\mu}}|\Omega|>0\right\}.
\end{align*}

Now we are at the position to state our main results, which can be summarized into the following two theorems.

\begin{theorem}\label{th2.1}(Existence of a local minimum solution).
Assume that $(b, \lambda,\mu)\in \bigcup\limits_{i=1}^6 \mathcal{A}_i$. Set
$$c_{\rho}:=\inf_{u\in B_{\rho}}I(u),$$
where $B_{\rho}:=\{u\in H_0^1(\Omega): \|u\|\leq \rho\}$ and $\rho>0$ is given in
Lemma \ref{le3.1}. Then problem \eqref{eq1} admits a local minimum solution $u$
with $I(u)=c_\rho<0$.
\end{theorem}

\begin{theorem}\label{th2.2}(Existence of a least energy solution).
Assume that $(b, \lambda,\mu)\in \bigcup\limits_{i=1}^6 \mathcal{A}_i$. When $(b, \lambda,\mu)\in \mathcal{A}_5\cup\mathcal{A}_6$, we suppose in addition that $2bS^2-1\geq0$.
Set $$c_{\mathcal{K}}:=\inf_{u\in \mathcal{K}}I(u),$$
where $\mathcal{K}=\{u\in H_0^1(\Omega):I'(u)=0\}$.
Then problem \eqref{eq1} admits a least energy solution $u$ with $I(u)=c_{\mathcal{K}}<0$.
\end{theorem}

\begin{remark}\label{remark-1}
From Theorems \ref{th2.1} and \ref{th2.2}, it is easy to see that $c_{\rho}\geq c_{\mathcal{K}}$. On the other hand, if $u$ is a least energy solution to problem \eqref{eq1} obtained in Theorem \ref{th2.2}, one obtains that
\begin{align*}
0>c_\mathcal{K}=I(u)-\frac{1}{4}\langle I'(u),u\rangle=&\frac{1}{4}\|u\|^{2}-\frac{\lambda}{4}\|u\|_{2}^{2} +\frac{\mu}{2}\|u\|_{2}^{2}-\frac{\mu}{4}\int_{\Omega}u^{2}\ln u^{2}\mathrm{d}x\\
&\geq\frac{1}{4}\|u\|^{2}+\frac{\mu}{4}e^{1-\frac{\lambda}{\mu}}|\Omega|.
\end{align*}
The last inequality comes from a similar estimate to that in \eqref{logineq1}. Then, we have
$\|u\|^2<|\mu|e^{1-\frac{\lambda}{\mu}}|\Omega|$. Under the assumptions of Theorems \ref{th2.1} and \ref{th2.2}, if we further assume that
$$|\mu|e^{1-\frac{\lambda}{\mu}}|\Omega| \leq \rho^2,$$
where $\rho>0$ is given in Lemma \ref{le3.1}, then we have $\|u\|\leq \rho$, which implies
$c_{\rho}\leq c_{\mathcal{K}}$. Therefore, $c_{\rho}= c_{\mathcal{K}}$, and the local minimum solution is also a least energy solution to problem \eqref{eq1}.
\end{remark}


The existence of weak solutions to problem \eqref{eq1} given by Theorems \ref{th2.1} and \ref{th2.2} can be
described on the following $(\lambda,\mu)$ plane by Figure \ref{picture}, where the pink regions stand for
the existence of nontrivial solutions, $\eta_1$ and $\eta_2$ are defined as follows
\begin{align*}
&\eta_1 : \frac{(\lambda_1(\Omega)-\lambda)^2S^2}{4\lambda^2_1(\Omega)(1-bS^2)}+\frac{\mu}{2}|\Omega|=0,
&\eta_2 : \frac{S^2}{4(1-bS^2)}+\frac{\mu}{2}e^{-\frac{\lambda}{\mu}}|\Omega|=0.
\end{align*}

\begin{figure}[H]
\begin{minipage}[t]{0.5\linewidth}
\centering
\includegraphics[width=6cm,height=4cm]{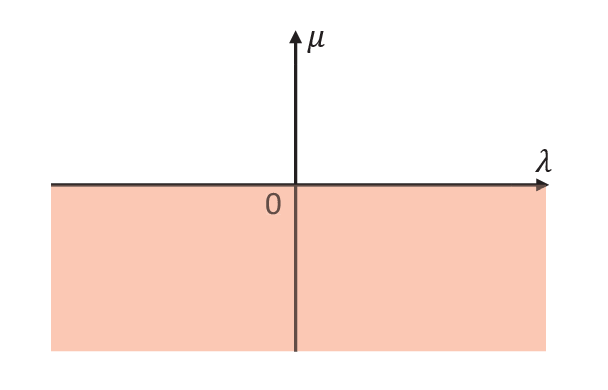}
\caption*{$(1) \ b\geq\frac{1}{S^2}$}\label{picture1}
\end{minipage}
\begin{minipage}[t]{0.5\linewidth}
\centering
\includegraphics[width=6cm,height=4cm]{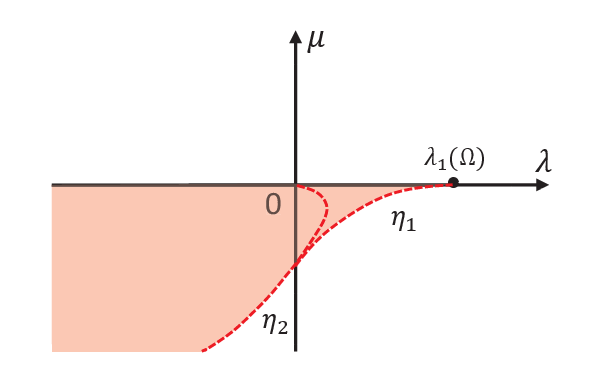}
\caption*{$(2) \ 0<b<\frac{1}{S^2}$ or $\frac{1}{2S^2}\leq b<\frac{1}{S^2}$}\label{picture2}
\end{minipage}
\caption{The existence of solutions}\label{picture}
\end{figure}

\begin{remark}\label{compare-1}
For the case $b=0$, Hajaiej et al. \cite{Hajaiej2024} showed that when $(\lambda,\mu)\in \mathcal{B}_1\cup\mathcal{B}_2$,
problem \eqref{eq1} admits a nontrivial solution $u$, which is a local minimum of the energy functional or a least energy solution.
Here
\begin{align*}
\mathcal{B}_1&=\left\{(\lambda, \mu)|\ \lambda\in \left[0, \lambda_1(\Omega)\right),\mu<0,
 \left(\frac{\lambda_{1}(\Omega)-\lambda}{\lambda_{1}(\Omega)}\right)
 ^2S^2+\frac{\mu}{2}|\Omega|>0\right\},\\
\mathcal{B}_2&=\left\{(\lambda, \mu)|\ \lambda\in \mathbb{R}, \mu<0,
 S^2+2\mu e^{-\frac{\lambda}{\mu}}|\Omega|>0\right\}.
\end{align*}
It is directly verified that when $0<b<1/S^2$, $(\lambda,\mu)\in \mathcal{B}_1$ implies $(b,\lambda,\mu)\in \mathcal{A}_5$,
while $(\lambda,\mu)\in \mathcal{B}_2$ implies $(b,\lambda,\mu)\in \mathcal{A}_6$. Moreover, when $b\geq 1/S^2$,
problem \eqref{eq1} admits a nontrivial solution for all $(\lambda,\mu)$ with $\mu<0$ and $\lambda\in \mathbb{R}$.
This means that the presence of the nonlocal term enlarges the range of the parameters $(\lambda,\mu)$ for problem
\eqref{eq1} to admit a nontrivial weak solution.
\end{remark}

\begin{remark}\label{compare-2}
Quite recently, the second author of this paper \cite{Han24} studied problem \eqref{eq1} with $\lambda=0$, $\mu>0$
and $bS^2<1<2bS^2$. By combining a result by Jeanjean \cite{J99} and a recent estimate by Deng et al. \cite{DengHe2023} with the Mountain Pass Lemma and Br\'{e}zis-Lieb's lemma, he proved that either the norm of the sequence of approximated solution goes to infinity or the original problem admits a mountain pass type solution. Furthermore, the former case can be excluded when $\Omega$ is star-shaped. In \cite{Han24} the condition $bS^2<1<2bS^2$ (which means that $b>0$ can be neither too large nor too small) plays a key role for the Palais-Smale sequence of the approximated problem to have a strongly convergent subsequence. The main differences between the two papers lie in the following three aspects: the first one is that Theorem \ref{th2.1} shows that for problem \eqref{eq1} to admit a nontrivial solution we allow $b$ to be any positive number; the second one is that the energy level of the solution obtained here is negative and the third one is that the existence result in this paper is valid for any bounded smooth domain.
\end{remark}

\par
\section{Proofs of the main results}
\setcounter{equation}{0}

The proof of the main results is based on several lemmas. We first verify that there are
two positive constants $\alpha$ and $\rho$ such that the energy functional $I(u)$
satisfies $I(u)\geq \alpha$ for all $u\in H_0^1(\Omega)$ with $\|u\|=\rho$.

\begin{lemma}\label{le3.1}
Assume that $(b, \lambda,\mu)\in \bigcup\limits_{i=1}^6 \mathcal{A}_i$.
Then there exist $\alpha,\rho>0$ such that $I(u)\geq \alpha$ for all $u\in H_0^1(\Omega)$ with $\|u\|=\rho$.
\end{lemma}

\begin{proof}
The proof is divided into two cases.

\noindent{\bf Case 1:} $(b, \lambda,\mu)\in \mathcal{A}_1\cup\mathcal{A}_3\cup\mathcal{A}_5$.

For all $u\in H_0^1(\Omega)\backslash \{0\}$, since $\mu<0$, by using \eqref{ineq1}, \eqref{log2} and recalling
the definition of $\lambda_1(\Omega)$, one has
\begin{align}\label{inequa1}
I(u)&=\frac{1}{2}\|u\|^2+\frac{b}{4}\|u\|^4-\frac{\lambda}{2}\|u\|_2^2+\frac{\mu}{2}\|u\|_2^2-\frac{\mu}{2}\int_{\Omega}u^2\ln{u^2}\mathrm{d}x
-\frac{1}{4}\|u\|_4^4\nonumber\\
&=\frac{1}{2}\|u\|^2+\frac{b}{4}\|u\|^4-\frac{\lambda}{2}\|u\|_2^2+\frac{\mu}{2}\int_{\Omega}u^2\left(1-\ln{u^2}\right)\mathrm{d}x-\frac{1}{4}\|u\|_4^4\nonumber\\
&\geq \frac{1}{2}\|u\|^2+\frac{b}{4}\|u\|^4-\frac{\lambda}{2}\|u\|_2^2+\frac{\mu}{2}|\Omega|-\frac{1}{4}\|u\|_4^4\nonumber\\
&\geq \frac{1}{2}\|u\|^2+\frac{b}{4}\|u\|^4-\frac{\lambda}{2\lambda_1(\Omega)}\|u\|^2+\frac{\mu}{2}|\Omega|-\frac{1}{4S^2}\|u\|^4\nonumber\\
&=\frac{\lambda_1(\Omega)-\lambda}{2\lambda_1(\Omega)}\|u\|^2+\frac{bS^2-1}{4S^2}\|u\|^4+\frac{\mu}{2}|\Omega|.
\end{align}
Set
\begin{align*}
g(t)=\frac{\lambda_1(\Omega)-\lambda}{2\lambda_1(\Omega)}t^2+\frac{bS^2-1}{4S^2}t^4+\frac{\mu}{2}|\Omega|, \ \ t>0.
\end{align*}
Then
\begin{align*}
g'(t)=t\left(\frac{bS^2-1}{S^2}t^2+\frac{\lambda_1(\Omega)-\lambda}{\lambda_1(\Omega)}\right), \ \ t>0.
\end{align*}

When $(b, \lambda,\mu)\in \mathcal{A}_1\cup\mathcal{A}_3$, it is directly verified from \eqref{inequa1} that
there exist $\alpha,\rho>0$ such that $I(u)\geq \alpha$ for all $\|u\|=\rho$.

When $(b, \lambda,\mu)\in \mathcal{A}_5$, $g(t)$ takes its maximum at $t_g:=\left(\frac{(\lambda_{1}(\Omega)-\lambda)S^2}{\lambda_{1}(\Omega)(1-bS^2)}\right)^{\frac{1}{2}}$
and
\begin{eqnarray*}
g(t_g)=\frac{(\lambda_1(\Omega)-\lambda)^2S^2}{4\lambda^2_1(\Omega)(1-bS^2)}+\frac{\mu}{2}|\Omega|>0,
\end{eqnarray*}
which ensures the conclusion of this lemma with $\alpha=g(t_g)$ and $\rho= t_g$.

\noindent{\bf Case 2:} $(b, \lambda,\mu)\in \mathcal{A}_2\cup\mathcal{A}_4\cup\mathcal{A}_6$.

For all $u\in H_0^1(\Omega)\backslash \{0\}$, since $\mu<0$, it follows from \eqref{log1} that
\begin{align}\label{logineq1}
&-\frac{\lambda}{2}\|u\|_{2}^{2}+\frac{\mu}{2}\|u\|_{2}^{2}-\frac{\mu}{2}\int_{\Omega}u^{2}\ln u^{2}\mathrm{d}x\nonumber\\
=&-\frac{\mu}{2}\int_{\Omega}u^{2}\left(\frac{\lambda}{\mu}-1+\ln u^{2}\right)\mathrm{d}x\nonumber\\
=&-\frac{\mu}{2}e^{1-\frac{\lambda}{\mu}}\int_{\Omega}e^{\frac{\lambda}{\mu}-1}u^{2}\ln \left(e^{\frac{\lambda}{\mu}-1}u^{2}\right)\mathrm{d}x\nonumber\\
=&-\frac{\mu}{2}e^{1-\frac{\lambda}{\mu}}\int_{\{x\in\Omega:\,e^{\frac{\lambda}{\mu}-1}u^{2}(x)>1\}}
e^{\frac{\lambda}{\mu}-1}u^{2}\ln\left(e^{\frac{\lambda}{\mu}-1}u^{2}\right)\mathrm{d}x\nonumber\\
&-\frac{\mu}{2}e^{1-\frac{\lambda}{\mu}}\int_{\{x\in\Omega:\,e^{\frac{\lambda}{\mu}-1}u^{2}\leq1\}}e^{\frac{\lambda}{\mu}-1}u^{2}
\ln\left(e^{\frac{\lambda}{\mu}-1}u^{2}\right)\mathrm{d}x\nonumber\\
\geq&-\frac{\mu}{2}e^{1-\frac{\lambda}{\mu}}\int_{\{x\in\Omega:\,e^{\frac{\lambda}{\mu}-1}u^{2}\leq1\}}e^{\frac{\lambda}{\mu}-1}u^{2}
\ln\left(e^{\frac{\lambda}{\mu}-1}u^{2}\right)\mathrm{d}x\nonumber\\
\geq& \ \frac{\mu}{2}e^{-\frac{\lambda}{\mu}}|\Omega|.
\end{align}
Consequently, in view of $\mu<0$, \eqref{ineq1} and \eqref{logineq1}, one has
\begin{align}\label{inequa2}
I(u)&=\frac{1}{2}\|u\|^2+\frac{b}{4}\|u\|^4-\frac{\lambda}{2}\|u\|_2^2+\frac{\mu}{2}\|u\|_2^2-\frac{\mu}{2}\int_{\Omega}u^2\ln{u^2}\mathrm{d}x
-\frac{1}{4}\|u\|_4^4\nonumber\\
&=\frac{1}{2}\|u\|^2+\frac{b}{4}\|u\|^4-\frac{\mu}{2}\int_{\Omega}u^2\left(\frac{\lambda}{\mu}-1+\ln{u^2}\right)\mathrm{d}x-\frac{1}{4}\|u\|_4^4\nonumber\\
&\geq \frac{1}{2}\|u\|^2+\frac{b}{4}\|u\|^4+\frac{\mu}{2}e^{-\frac{\lambda}{\mu}}|\Omega|-\frac{1}{4S^2}\|u\|^4\nonumber\\
&=\frac{1}{2}\|u\|^2+\frac{bS^2-1}{4S^2}\|u\|^4+\frac{\mu}{2}e^{-\frac{\lambda}{\mu}}|\Omega|.
\end{align}
Set
\begin{align*}
h(t)=\frac{1}{2}t^2+\frac{bS^2-1}{4S^2}t^4+\frac{\mu}{2}e^{-\frac{\lambda}{\mu}}|\Omega|, \ \ t>0.
\end{align*}
Then
\begin{align*}
h'(t)=t\left(1+\frac{bS^2-1}{S^2}t^2\right), \ \ t>0.
\end{align*}

Similar to the argument in Case 1, when $(b, \lambda,\mu)\in\mathcal{A}_2\cup\mathcal{A}_4$,
we see from \eqref{inequa2} that there exist two positive constants $\alpha$ and $\rho$ such
that $I(u)\geq \alpha$ for all $\|u\|=\rho$. When $(b, \lambda,\mu)\in \mathcal{A}_6$,
$h(t)$ takes its maximum at $t_h:=\left(\frac{S^2}{1-bS^2}\right)^{\frac{1}{2}}$
and
\begin{eqnarray*}
h(t_h)=\frac{S^2}{4(1-bS^2)}+\frac{\mu}{2}e^{-\frac{\lambda}{\mu}}|\Omega|>0.
\end{eqnarray*}
Consequently, by taking $\alpha=h(t_h)$ and $\rho= t_h$ one sees that $I(u)\geq\alpha$
for all $\|u\|=\rho$. The proof is complete.
\end{proof}

Next, we give some characterizations of the variational constants $c_\rho$ and $c_{\mathcal{K}}$.
\begin{lemma}\label{le3.2}
Assume that $(b, \lambda,\mu)\in \bigcup\limits_{i=1}^6 \mathcal{A}_i$.
Let $c_\rho$ and $c_{\mathcal{K}}$ be the variational constants given in Theorems \ref{th2.1} and
\ref{th2.2}, respectively. Then $-\infty<c_\rho,\ c_{\mathcal{K}}<0$.
\end{lemma}

\begin{proof}
(I) We first prove $-\infty<c_\rho<0$. When $(b, \lambda,\mu)\in \bigcup\limits_{i=1}^6 \mathcal{A}_i$,
it follows from \eqref{inequa1} and \eqref{inequa2} that, for any $u\in H_0^1(\Omega)$
\begin{align}\label{inequation3}
I(u)\geq\frac{\lambda_1(\Omega)-\lambda}{2\lambda_1(\Omega)}\|u\|^2+\frac{bS^2-1}{4S^2}\|u\|^4+\frac{\mu}{2}|\Omega|, \qquad (b, \lambda,\mu)\in \mathcal{A}_1\cup\mathcal{A}_3\cup \mathcal{A}_5,
\end{align}
or
\begin{align}\label{inequation4}
I(u)\geq\frac{1}{2}\|u\|^2+\frac{bS^2-1}{4S^2}\|u\|^4+\frac{\mu}{2}e^{-\frac{\lambda}{\mu}}|\Omega|,\qquad (b, \lambda,\mu)\in \mathcal{A}_2\cup\mathcal{A}_4\cup\mathcal{A}_6.
\end{align}
According to \eqref{inequation3} and \eqref{inequation4}, one has $I(u)>-\infty$ when $\|u\|\leq\rho$. Thus, $c_\rho>-\infty$.

To show that $c_\rho<0$, for any $u\in H_0^1(\Omega)\setminus\{0\}$,
consider $I(tu)$ with $t\geq0$. By the definition of $I(u)$ we know
\begin{align*}
I(tu)=t^2\left(\frac{1}{2}\| u\|^{2}+\frac{b}{4}t^2\|u\|^4-\frac{\lambda}{2}\|u\|_{2}^{2}-\frac{\mu}{2}\ln t^2\|u\|_{2}^{2}
+\frac{\mu}{2}\int_{\Omega}u^{2}(1-\ln u^{2})\mathrm{d}x-\frac{1}{4}t^2\|u\|_{4}^{4}\right).
\end{align*}
Noticing that $\mu<0$, it can be deduced from the above equality that there exists a $t_u>0$ suitably
small such that $\|t_u u\|\leq \rho$ and $I(t_uu)<0$, which implies that $c_\rho<0$.

(II) Next we show that $-\infty<c_{\mathcal{K}}<0$. Notice that $u\in\mathcal{K}$,
where $u$ is the solution obtained in Theorem \ref{th2.1}. Hence $c_\mathcal{K}\leq c_\rho<0$.
It remains to prove $c_{\mathcal{K}}>-\infty$ to complete the proof.

For any $u\in\mathcal{K}$, one obtains, with the help of \eqref{functional}, \eqref{frechet}
and \eqref{log1}, that
\begin{align*}
I(u)&=I(u)-\frac{1}{4}\langle I'(u),u\rangle\\
&=\frac{1}{4}\|u\|^{2}-\frac{\lambda}{4}\|u\|_{2}^{2} +\frac{\mu}{2}\|u\|_{2}^{2}-\frac{\mu}{4}\int_{\Omega}u^{2}\ln u^{2}\mathrm{d}x\\
&\geq-\frac{\mu}{4}\int_{\Omega}u^{2}\left(\frac{\lambda}{\mu}-2+\ln {u^2}\right)\mathrm{d}x\\
&=-\frac{\mu}{4}\int_{\Omega}u^{2}\ln \left(u^2e^{\frac{\lambda}{\mu}-2}\right)\mathrm{d}x\\
&=-\frac{\mu}{4}e^{2-\frac{\lambda}{\mu}}\int_{\Omega}u^{2}e^{\frac{\lambda}{\mu}-2}\ln \left(u^2e^{\frac{\lambda}{\mu}-2}\right)\mathrm{d}x\\
&\geq-\frac{\mu}{4}e^{2-\frac{\lambda}{\mu}}\int_{\{x\in\Omega:\ e^{\frac{\lambda}{\mu}-2}u^{2}\leq1\}}
u^{2}e^{\frac{\lambda}{\mu}-2}\ln \left(u^2e^{\frac{\lambda}{\mu}-2}\right)\mathrm{d}x\\
&\geq\frac{\mu}{4}e^{1-\frac{\lambda}{\mu}}|\Omega|,
\end{align*}
which implies that $c_{\mathcal{K}}>-\infty$. The proof is complete.
\end{proof}

The following compactness result, which is proved by combining Eklend's variational principle
with Brezis-Lieb's Lemma, plays a key role in proving the main results.

\begin{lemma}\label{le3.3}
Assume that $(b, \lambda,\mu)\in \bigcup\limits_{i=1}^6 \mathcal{A}_i$.
Then there exists a minimizing sequence $\{u_n\}$ for $c_\rho$ and a $u\in H_0^1(\Omega)$ such that, up to a subsequence, $u_n \rightarrow u$
in $H_0^1(\Omega)$ as $n\rightarrow\infty$, where $c_{\rho}:=\inf_{\|u\|\leq\rho}I(u)$.
\end{lemma}

\begin{proof}
The technique of the proof is inspired by \cite{ZhangHanWang2023}, with some convergence tricks
to deal with the nonlocal term. For any $r>0$, set $B_r=\{u\in H_0^1(\Omega):\|u\|\leq r\}$.
Let $\{v_n\}\subset B_\rho$ be a minimizing sequence for $c_\rho$. Without
loss of generality, we may assume that $c_\rho\leq I(v_n)\leq c_\rho+\dfrac{1}{n}$ for all
$n\in\mathbb{N}$. Using the fact that $c_\rho<0$ and inequalities \eqref{inequa1} and \eqref{inequa2},
one knows that there is a constant $\tau>0$ suitably small such that $\|v_n\|\leq\rho-\tau$ for all $n$ large enough.
Moreover, in view of the continuity of $I(u)$, Lemmas \ref{Ekeland} and \ref{le3.2},
one sees that there exists another minimizing sequence $\{u_n\}\subset B_\rho$ for $c_\rho$ such that
\begin{align}\label{eke-ineq}
&I(u_n)\leq I(v_n),\nonumber\\
&\|u_n-v_n\|\leq\frac{\tau}{2},\nonumber\\
I(w)>&I(u_n)-\frac{2}{n\tau} \|w-u_n\|,\qquad\forall\ w\in B_\rho,\ w\neq u_n.
\end{align}
It is obvious that $\|u_n\|\leq \rho-\frac{\tau}{2}$.

For any $h\in H_0^1(\Omega)$ with $\|h\|=1$, set $z(t)=u_n+th$, $t\in(0,\frac{\tau}{2}]$.
Then $\|z(t)\|\leq \rho,\ \forall\ t\in(0,\frac{\tau}{2}]$. Taking $w=z(t)$ in \eqref{eke-ineq} implies
\begin{align}\label{eke-ineq2}
\frac{I(z(t))-I(u_n)}{t}>-\frac{2}{n\tau}.
\end{align}
Letting $t\rightarrow0^+$ in \eqref{eke-ineq2}, we arrive at
\begin{align}\label{fd1}
\langle I'(u_n), h \rangle\geq-\frac{2}{n\tau}.
\end{align}
Replacing $h$ by $-h$ in \eqref{fd1}, one has
\begin{align*}
\langle I'(u_n), h \rangle\leq\frac{2}{n\tau},
\end{align*}
which, together with \eqref{fd1} and the arbitrariness of $h$,
implies
\begin{equation}\label{approximated}
\lim_{n\rightarrow\infty}I'(u_n)=0.
\end{equation}

On the other hand, by the boundedness of $\{u_n\}$ we know that there is a subsequence
of $\{u_n\}$ (which we still denote by $\{u_n\}$) and a $u\in H_0^1(\Omega)$ such that,
as $n\rightarrow\infty$,
\begin{equation}\label{convergence}
\begin{cases}
u_{n}\rightharpoonup u \ in \ H_{0}^1(\Omega),\\
u_{n}\rightarrow u \ in \ L^p(\Omega), \ \ 1\leq p<4,\\
|u_n|^2u_n\rightharpoonup |u|^2u \ in \ L^{\frac{4}{3}}(\Omega),\\
u_{n}\rightarrow u \ \ a.e.\ in \ \Omega.
\end{cases}
\end{equation}
We claim that
\begin{align}\label{log-convergence1}
\lim_{n\rightarrow\infty}\int_{\Omega}u_n^{2}\ln u_n^2\mathrm{d}x
=\int_{\Omega}u^{2}\ln u^2\mathrm{d}x,
\end{align}
and
\begin{align}\label{log-convergence2}
\lim_{n\rightarrow\infty}\int_{\Omega}u_n\phi\ln u_n^2\mathrm{d}x
=\int_{\Omega}u\phi\ln u^2\mathrm{d}x,\qquad \forall\ \phi\in H_0^1(\Omega).
\end{align}
Indeed, since $u_n\rightarrow u$ a.e. in $\Omega$ as $n\rightarrow\infty$, there holds
\begin{align}\label{dui-1}
u_n^2\ln u_n^2\rightarrow u^2\ln u^2 \ \ a.e. \ in\ \Omega \ \  as \ n\rightarrow\infty.
\end{align}
Applying \eqref{log1}, \eqref{log4} with $\delta=\frac{1}{2}$ and recalling \eqref{convergence}, we have
\begin{align}\label{dui-2}
|u_n^2\ln u_n^2|
\leq\dfrac{1}{e}+\frac{2}{e}u_n^3\rightarrow\dfrac{1}{e}+\frac{2}{e}u^3 \ in\ L^1(\Omega)\ as\ n\rightarrow\infty.
\end{align}
Consequently, combining \eqref{dui-1} and \eqref{dui-2} with Lebesgue's dominated convergence theorem,
we obtain \eqref{log-convergence1}. \eqref{log-convergence2} follows in quite a similar way.

It remains to show that $u_n\rightarrow u \ in \ H_0^1(\Omega)$ as $n\rightarrow\infty$
to complete the proof. For this, set $w_n=u_n-u$. Then $\{w_n\}$ is also a bounded
sequence in $H_0^1(\Omega)$. Hence there exists a subsequence of $\{w_{n}\}$ (still denoted by $\{w_{n}\}$) such that
\begin{align}\label{wn}
\lim\limits_{n\rightarrow\infty}\|w_n\|^2=l\geq0.
\end{align}
On the one hand, by the weak convergence $u_{n}\rightharpoonup u$ in $H_{0}^1(\Omega)$, we have
\begin{align}\label{wn1}
\|u_n\|^{2}=\| w_n\|^2+\| u\|^{2}+o_n(1), \qquad  n\rightarrow\infty.
\end{align}
On the other hand, it follows from the boundedness of $\{u_n\}$ in $H_0^1(\Omega)$ and Sobolev
embedding theorem that $\|u_n\|_4\leq C$, which, together with $u_n\rightarrow u \ a.e. \ in \ \Omega$
and Lemma \ref{lem-Lieb}, implies that
\begin{align}\label{wn2}
\|u_n\|_4^4=\|w_n\|_4^4+\|u\|_4^4+o_n(1),\qquad n\rightarrow\infty.
\end{align}
Since $I'(u_n)\rightarrow0$ as $n\rightarrow\infty$, by using \eqref{convergence} and \eqref{log-convergence2}, we obtain, as $n\rightarrow\infty$,
\begin{align}\label{Iun-u}
o_n(1)=&\langle I'(u_n),u\rangle\nonumber\\
=&\left(1+b\|u_n\|^2\right)\int_{\Omega}\nabla u_n\nabla u\mathrm{d}x-\lambda\int_{\Omega}u_n u\mathrm{d}x-\mu\int_{\Omega}u_n u\ln u_n^2\mathrm{d}x-
\int_{\Omega}|u_n|^2u_n u\mathrm{d}x\nonumber\\
=&\|u\|^2+b\|u_n\|^2\|u\|^2-\lambda\|u\|_2^2-\mu\int_{\Omega}u^2\ln u^2\mathrm{d}x-\|u\|_4^4+o_n(1)\nonumber\\
=&\|u\|^2+b\|u\|^4+b\|w_n\|^2\|u\|^2-\lambda\|u\|_2^2-\mu\int_{\Omega}u^2\ln u^2\mathrm{d}x-\|u\|_4^4+o_n(1).
\end{align}
Then, by means of \eqref{convergence}, \eqref{log-convergence1}, \eqref{wn1}, \eqref{wn2} and \eqref{Iun-u}, we get
\begin{align}\label{Iun-un}
o_n(1)=&\langle I'(u_n),u_n\rangle\nonumber\\
=&\left(1+b\|u_n\|^2\right)\|u_n\|^2-\lambda \|u_n\|_2^2-\mu\int_{\Omega}u_n^2\ln u_n^2\mathrm{d}x-\|u_n\|_4^4\nonumber\\
=&\|u_n\|^2+b\|u_n\|^4-\lambda\|u\|_2^2-\mu\int_{\Omega}u^2\ln u^2\mathrm{d}x-\|u_n\|_4^4+o_n(1)\nonumber\\
=&\|w_n\|^2+\|u\|^2+b\|w_n\|^4+2b\|w_n\|^2\|u\|^2+b\|u\|^4-\lambda\|u\|_2^2-\mu\int_{\Omega}u^2\ln u^2\mathrm{d}x\nonumber\\
&-\|w_n\|_4^4-\|u\|_4^4+o_n(1)\nonumber\\
=&\langle I'(u_n),u\rangle+\|w_n\|^2+b\|w_n\|^4+b\|w_n\|^2\|u\|^2-\|w_n\|_4^4\nonumber\\
=&\|w_n\|^2+b\|w_n\|^4+b\|w_n\|^2\|u\|^2-\|w_n\|_4^4+o_n(1).
\end{align}
It follows from \eqref{ineq1} and  \eqref{Iun-un} that
\begin{align}\label{Iun-un-ineq}
o_n(1)\geq\|w_n\|^2+b\|w_n\|^4+b\|w_n\|^2\|u\|^2-\frac{1}{S^2}\|w_n\|^4.
\end{align}
Letting $n\rightarrow\infty$ in \eqref{Iun-un-ineq} and recalling \eqref{wn}, one gets
\begin{align}\label{inequation5}
l\left(1+b\|u\|^2+\frac{(bS^2-1)l}{S^2}\right)\leq0.
\end{align}

When $(b, \lambda,\mu)\in \mathcal{A}_1\cup\mathcal{A}_2\cup\mathcal{A}_3\cup\mathcal{A}_4$,
it can be directly checked from \eqref{inequation5} that $l=0$, i.e., $u_n\rightarrow u$ in
$H_0^1(\Omega)$ as $n\rightarrow\infty$.

When $(b, \lambda,\mu)\in \mathcal{A}_5\cup\mathcal{A}_6$, we assume by contradiction that $l>0$.
Then according to \eqref{convergence}, \eqref{log-convergence1}, \eqref{wn1}, \eqref{wn2} and \eqref{Iun-un}, we have
\begin{align}\label{Iun-Iu}
I(u_n)=&\frac{1}{2}\|u_n\|^2+\frac{b}{4}\|u_n\|^4-\frac{\lambda}{2}\|u_n\|_2^2+\frac{\mu}{2}\|u_n\|_2^2-\frac{\mu}{2}\int_{\Omega}u_n^2\ln{u_n^2}\mathrm{d}x
-\frac{1}{4}\|u_n\|_4^4\nonumber\\
=&\frac{1}{2}\|w_n\|^2+\frac{1}{2}\|u\|^2+\frac{b}{4}\|w_n\|^4+\frac{b}{4}\|u\|^4+\frac{b}{2}\|w_n\|^2\|u\|^2-\frac{\lambda}{2}\|u\|_2^2+\frac{\mu}{2}\|u\|_2^2\nonumber\\
&-\frac{\mu}{2}\int_{\Omega}u^2\ln{u^2}\mathrm{d}x-\frac{1}{4}\|w_n\|_4^4-\frac{1}{4}\|u\|_4^4+o_n(1)\nonumber\\
=&I(u)+\frac{1}{2}\|w_n\|^2+\frac{b}{4}\|w_n\|^4+\frac{b}{2}\|w_n\|^2\|u\|^2-\frac{1}{4}\|w_n\|_4^4+o_n(1)\nonumber\\
=&I(u)+\frac{1}{4}\langle I'(u_n),u_n\rangle+\frac{1}{4}\|w_n\|^2+\frac{b}{4}\|w_n\|^2\|u\|^2+o_n(1)\nonumber\\
=&I(u)+\frac{1}{4}\|w_n\|^2+\frac{b}{4}\|w_n\|^2\|u\|^2+o_n(1).
\end{align}
Letting $n\rightarrow\infty$ in \eqref{Iun-Iu} yields
\begin{align}\label{equation1}
c_\rho=\lim\limits_{n\rightarrow\infty}I(u_n)=I(u)+\frac{1}{4}l+\frac{b}{4}l\|u\|^2>I(u).
\end{align}
On the other hand, since $\|u_n\|<\rho$, it follows from the weak lower semi-continuity of the norm that $\|u\|\leq\rho$,
which implies  $I(u)\geq c_\rho$. This contradicts \eqref{equation1}. Hence $l=0$, i.e., $u_n\rightarrow u$ in $H_0^1(\Omega)$ as $n\rightarrow\infty$. This completes the proof.
\end{proof}

With the above lemmas at hand, we are now able to prove Theorems \ref{th2.1} and \ref{th2.2}.

{\bf Proof of Theorem \ref{th2.1}.} From Lemma \ref{le3.3} we know that there exists
a minimizing sequence $\{u_n\}\subset B_{\rho-\tau/2}$ for $c_\rho$ and a $u\in H_0^1(\Omega)$
such that $u_n \rightarrow u$ in $H_0^1(\Omega)$ as $n\rightarrow\infty$. Consequently,
$\|u\|\leq \rho-\tau/2$. Noticing that $I(u)$ is a $\mathcal{C}^1$ functional in $H_0^1(\Omega)$
and recalling \eqref{approximated}, one obtains
\begin{equation*}
\begin{cases}
I(u)=\lim\limits_{n\rightarrow\infty}I(u_n)=c_\rho,&\\
I'(u)=\lim\limits_{n\rightarrow\infty}I'(u_n)=0.&
\end{cases}
\end{equation*}
This means that $u$ is a weak solution to problem \eqref{eq1} which is also a local
minimum of the functional $I(u)$. This completes the proof of Theorem \ref{th2.1}.

{\bf Proof of Theorem \ref{th2.2}.}
From Theorem \ref{th2.1} we know that $\mathcal{K}$ is no-empty. Consequently, $c_\mathcal{K}$ is well defined by Lemma \ref{le3.2}.
Take a minimizing sequence $\{u_n\}\subset\mathcal{K}$ for $c_\mathcal{K}$. Then it is obvious that $I'(u_n)=0$.
This, together with $I(u_n)\rightarrow c_\mathcal{K}$ as $n\rightarrow \infty$ and \eqref{log1},  shows that, for $n$ suitably large,
\begin{align*}
c_\mathcal{K}+1+o_n(1)\|u_n\|&\geq I(u_n)-\frac{1}{4}\langle I'(u_n),u_n\rangle\nonumber\\
&=\frac{1}{4}\|u_n\|^{2}-\frac{\lambda}{4}\|u_n\|_{2}^{2} +\frac{\mu}{2}\|u_n\|_{2}^{2}-\frac{\mu}{4}\int_{\Omega}u_n^{2}\ln u_n^{2}\mathrm{d}x\\
&=\frac{1}{4}\|u_n\|^{2}-\frac{\mu}{4}\int_{\Omega}u_n^{2}\left(\frac{\lambda}{\mu}-2+\ln {u_n^2}\right)\mathrm{d}x\\
&=\frac{1}{4}\|u_n\|^{2}-\frac{\mu}{4}e^{2-\frac{\lambda}{\mu}}\int_{\Omega}u_n^{2}e^{\frac{\lambda}{\mu}-2}
\ln\left(u_n^2e^{\frac{\lambda}{\mu}-2}\right)\mathrm{d}x\\
&\geq\frac{1}{4}\|u_n\|^{2}-\frac{\mu}{4}e^{2-\frac{\lambda}{\mu}}\int_{\{x\in\Omega:\ e^{\frac{\lambda}{\mu}-2}u_n^{2}\leq1\}}
u_n^{2}e^{\frac{\lambda}{\mu}-2}\ln \left(u_n^2e^{\frac{\lambda}{\mu}-2}\right)\mathrm{d}x\\
&\geq\frac{1}{4}\|u_n\|^{2}+\frac{\mu}{4}e^{1-\frac{\lambda}{\mu}}|\Omega|,
\end{align*}
which implies that $\{u_n\}$ is bounded in $H_0^1(\Omega)$.
Similar to the proof of Lemma \ref{le3.3}, \eqref{convergence}, \eqref{log-convergence1} and  \eqref{log-convergence2}  remain valid here.
Let $w_n=u_n-u$. Analogously, \eqref{wn}, \eqref{wn1}, \eqref{wn2}, \eqref{Iun-un} and \eqref{inequation5} hold.

When $(b, \lambda,\mu)\in \mathcal{A}_1\cup\mathcal{A}_2\cup\mathcal{A}_3\cup\mathcal{A}_4$,
it is directly seen from \eqref{inequation5} that $l=0$, i.e., $u_n\rightarrow u$ in
$H_0^1(\Omega)$ as $n\rightarrow\infty$.

Now we assume $(b, \lambda,\mu)\in \mathcal{A}_5\cup\mathcal{A}_6$ and $2bS^2-1\geq0$.
Suppose by contradiction that $l>0$. Then from \eqref{inequation5} we have
\begin{align}\label{ineq-1}
l\geq \frac{(1+b\|u\|^2)S^2}{1-bS^2}.
\end{align}
By virtue of \eqref{convergence}, \eqref{log-convergence1}, \eqref{wn1} and \eqref{wn2}, one has, as $n\rightarrow\infty$,
\begin{align*}
I(u_n)=&\frac{1}{2}\|u_n\|^2+\frac{b}{4}\|u_n\|^4-\frac{\lambda}{2}\|u_n\|_2^2+\frac{\mu}{2}\|u_n\|_2^2-\frac{\mu}{2}\int_{\Omega}u_n^2\ln{u_n^2}\mathrm{d}x
-\frac{1}{4}\|u_n\|_4^4\\
=&\frac{1}{2}\|w_n\|^2+\frac{b}{4}\|w_n\|^4+\frac{b}{4}\|w_n\|^2\|u\|^2-\frac{1}{4}\|w_n\|_4^4\\
&+\frac{1}{2}\|u\|^2+\frac{b}{4}\|u\|^4+\frac{b}{4}\|w_n\|^2\|u\|^2-\frac{\lambda}{2}\|u\|_2^2+\frac{\mu}{2}\|u\|_2^2
-\frac{\mu}{2}\int_{\Omega}u^2\ln{u^2}\mathrm{d}x-\frac{1}{4}\|u\|_4^4+o_n(1).
\end{align*}
Set
\begin{align}\label{equa-2}
I_1:=\frac{1}{2}\|w_n\|^2+\frac{b}{4}\|w_n\|^4+\frac{b}{4}\|w_n\|^2\|u\|^2-\frac{1}{4}\|w_n\|_4^4+o_n(1),
\end{align}
and
\begin{align}\label{equa-3}
I_2:=\frac{1}{2}\|u\|^2+\frac{b}{4}\|u\|^4+\frac{b}{4}\|w_n\|^2\|u\|^2-\frac{\lambda}{2}\|u\|_2^2+\frac{\mu}{2}\|u\|_2^2
-\frac{\mu}{2}\int_{\Omega}u^2\ln{u^2}\mathrm{d}x-\frac{1}{4}\|u\|_4^4.
\end{align}
Now we estimate $I_1$ and $I_2$ from below separately. Firstly, it follows from \eqref{Iun-un} that
\begin{align}\label{ineq-2}
I_1=&\frac{1}{4}\langle I'(u_n),u_n\rangle+\frac{1}{4}\|w_n\|^2=\frac{1}{4}\|w_n\|^2+o_n(1),\qquad n\rightarrow\infty.
\end{align}
Letting $n\rightarrow\infty$ in \eqref{ineq-2} and using \eqref{wn} and \eqref{ineq-1}, we have
\begin{align}\label{ineq-3}
I_1\geq\frac{(1+b\|u\|^2)S^2}{4(1-bS^2)}.
\end{align}
As for $I_2$, in view of \eqref{ineq1}, one arrives at
\begin{align}\label{ineq-4}
I_2\geq\frac{1}{2}\|u\|^2+\frac{b}{4}\|u\|^4+\frac{b}{4}\|w_n\|^2\|u\|^2
-\frac{\lambda}{2}\|u\|_2^2+\frac{\mu}{2}\|u\|_2^2
-\frac{\mu}{2}\int_{\Omega}u^2\ln{u^2}\mathrm{d}x-\frac{1}{4S^2}\|u\|^4.
\end{align}
Since $b>0$, letting $n\rightarrow\infty$ in \eqref{ineq-4} and using \eqref{wn} and \eqref{ineq-1} again,
one obtains
\begin{align}\label{ineq-5-1}
I_2\geq\frac{\lambda_1(\Omega)-\lambda}{2\lambda_1(\Omega)}\|u\|^2+\frac{b}{4}\|u\|^4
+\frac{(1+b\|u\|^2)bS^2}{4(1-bS^2)}\|u\|^2
+ \frac{\mu}{2}|\Omega|-\frac{1}{4S^2}\|u\|^4,
\end{align}
when $(b, \lambda,\mu)\in \mathcal{A}_5$ and $2bS^2-1\geq0$; and
\begin{align}\label{ineq-5-2}
I_2\geq\frac{1}{2}\|u\|^2+\frac{b}{4}\|u\|^4+\frac{(1+b\|u\|^2)bS^2}{4(1-bS^2)}\|u\|^2
+ \frac{\mu}{2}e^{-\frac{\lambda}{\mu}}|\Omega|-\frac{1}{4S^2}\|u\|^4,
\end{align}
when $(b, \lambda,\mu)\in \mathcal{A}_6$ and $2bS^2-1\geq0$.

Then it follows from \eqref{ineq-3}, \eqref{ineq-5-1} and \eqref{ineq-5-2} that
\begin{align*}
I(u_n)\geq& \frac{(1+b\|u\|^2)S^2}{4(1-bS^2)}+
\frac{\lambda_1(\Omega)-\lambda}{2\lambda_1(\Omega)}\|u\|^2+\frac{b}{4}\|u\|^4
+\frac{(1+b\|u\|^2)bS^2}{4(1-bS^2)}\|u\|^2+ \frac{\mu}{2}|\Omega|-\frac{1}{4S^2}\|u\|^4\\
=&\frac{S^2}{4(1-bS^2)}+\frac{2bS^2}{4(1-bS^2)}\|u\|^2+\frac{\lambda_1(\Omega)-\lambda}{2\lambda_1(\Omega)}\|u\|^2
+\frac{2bS^2-1}{4(1-bS^2)S^2}\|u\|^4+ \frac{\mu}{2}|\Omega|\\
\geq&\frac{S^2}{4(1-bS^2)}+ \frac{\mu}{2}|\Omega|\\
>&0,
\end{align*}
when $(b, \lambda,\mu)\in \mathcal{A}_5$ and $2bS^2-1\geq0$; and
\begin{align*}
I(u_n)\geq& \frac{(1+b\|u\|^2)S^2}{4(1-bS^2)}+
\frac{1}{2}\|u\|^2+\frac{b}{4}\|u\|^4+\frac{(1+b\|u\|^2)bS^2}{4(1-bS^2)}\|u\|^2+ \frac{\mu}{2}e^{-\frac{\lambda}{\mu}}|\Omega|-\frac{1}{4S^2}\|u\|^4\\
=&\frac{S^2}{4(1-bS^2)}+\frac{2}{4(1-bS^2)}\|u\|^2+\frac{2bS^2-1}{4(1-bS^2)S^2}\|u\|^4+ \frac{\mu}{2}e^{-\frac{\lambda}{\mu}}|\Omega|\\
\geq&\frac{S^2}{4(1-bS^2)}+ \frac{\mu}{2}e^{-\frac{\lambda}{\mu}}|\Omega|\\
>&0,
\end{align*}
when $(b, \lambda,\mu)\in \mathcal{A}_6$ and $2bS^2-1\geq0$.
Since $I(u_n)\rightarrow c_{\mathcal{K}}$ as $n\rightarrow\infty$, one sees from the above two
inequalities that $c_{\mathcal{K}}\geq0$ when $(b, \lambda,\mu)\in \mathcal{A}_5\cup\mathcal{A}_6$
and $2bS^2-1\geq0$, which contradicts $c_{\mathcal{K}}<0$. Hence $l=0$, i.e., $u_n\rightarrow u$
in $H_0^1(\Omega)$ as $n\rightarrow\infty$ when $(b, \lambda,\mu)\in \mathcal{A}_5\cup\mathcal{A}_6$
and $2bS^2-1\geq0$.

In summary, there exist a subsequence of $\{u_n\}$ (which we still denote by $\{u_n\}$) and a $u\in H_0^1(\Omega)$ such that $u_n\rightarrow u$ in $H_0^1(\Omega)$ as $n\rightarrow \infty$. Recalling again that $I(u)$ is $\mathcal{C}^1$ in $H_0^1(\Omega)$, one obtains
\begin{equation*}
\begin{cases}
 I(u)=\lim\limits_{n\rightarrow\infty}I(u_n)=c_\mathcal{K},&\\
 I'(u)=\lim\limits_{n\rightarrow\infty}I'(u_n)=0.&
\end{cases}
\end{equation*}
Therefore, $u$ is a least energy solution to problem \eqref{eq1}.
The proof of Theorem \ref{th2.2} is complete.

\vskip5mm

{\bf Declarations}\\
The authors declare that there are no relevant financial or non-financial competing interests to report.\\

{\bf Data availability}\\
No data was used for the research described in the article.\\

{\bf Acknowledgements}\\
The authors would like to thank Professor Tianhao Liu in Institute of Applied Physics and Computational
Mathematics for some valuable discussions.

\end{document}